\documentclass[11pt]{article}

\usepackage{pgf,tikz}
\usepackage{mathrsfs}
\usetikzlibrary{arrows}

\usepackage{pgfplots}
\pgfplotsset{compat=newest}

\usepackage{graphicx} 
\graphicspath{{figures/}} 

\usepackage{fullpage, url,amsmath,amsfonts,amssymb,mathtools,mathrsfs,graphicx,  algorithm, float, sansmath,epstopdf,color,caption,enumitem,tabularx}
\usepackage[final]{pdfpages}
\usepackage{amsthm, cite}
\usetikzlibrary{automata,topaths}
\usetikzlibrary{decorations.pathreplacing,shapes.misc}
\usepackage{fancyhdr}

\def\beq{\begin{equation}}
\def\eeq{\end{equation}}
\def\baq{\begin{eqnarray}}
\def\eaq{\end{eqnarray}}
\def\baqn{\begin{eqnarray*}}
\def\eaqn{\end{eqnarray*}}

\newcommand{\ball}{\mathbb{B}}

\usepackage{multirow}
\usetikzlibrary{calc,arrows}
\theoremstyle{plain}
\newtheorem{definition}{Definition}
\newtheorem{remark}{Remark}

\newtheorem{example}{Example}
\newtheorem{theorem}{Theorem}

\newtheorem{corollary}[theorem]{Corollary}
\newtheorem{proposition}[theorem]{Proposition}

\usepackage{blindtext}

\usepackage[colorlinks,linkcolor=blue,citecolor=red]{hyperref}

\usepackage{xcolor, framed}

\newcommand{\R}{{\mathbb R}}
\newcommand{\N}{{\mathbb N}}

\newcommand{\interior}{{\rm int}\kern 0.06em}

\def\<{\langle}
\def\>{\rangle}
\newcommand\dom{{\rm dom}}%

\usepackage{ulem}

\usepackage{subcaption}

\usepackage{pict2e}

\if
{

\usepackage[pageref]{backref}
\renewcommand*{\backrefalt}[4]{%
\ifcase #1 %
(Not cited)%
\or
(Cited on p.~#2)%
\else
(Cited on pp.~#2)%
\fi
}

}
\fi

\begin{document}
\title{{Compact R-Continuity with Applications to Solving Inclusions and Convergence of Algorithms}}
\author{Ba Khiet Le \thanks{Analytical and Algebraic Methods in Optimization Research Group, Faculty of Mathematics and Statistics, Ton Duc Thang University, Ho Chi Minh City, Vietnam.
 E-mail: \texttt{lebakhiet@tdtu.edu.vn}} \qquad Boris S. Mordukhovich\thanks{Department of Mathematics and Center for Artificial Intelligence and Data Science, Wayne State University, Detroit, MI 48202, USA.
 E-mail: \texttt{aa1086@wayne.edu}. Research of this author was partly supported by the US National Science Foundation under grant DMS-2204519 and by the Australian Research Council under Discovery Project DP-190100555} \qquad Michel  A. Th\' era \thanks{Mathematics and Computer Science Department, University of Limoges, 123 Avenue Albert Thomas,
87060 Limoges CEDEX, France.  E-mail: \texttt{michel.thera@unilim.fr}} }
\maketitle

\begin{abstract}
This paper investigates the notion of compact R-continuity and its specifications for set-valued mappings between Banach spaces. We reveal several important properties of compact R-continuity in general settings and show that in finite dimensions, this notion is supported by the classical \L ojasiewicz  inequality for analytic functions. An application of compact R-continuity and the obtained results {is} given to convergence analysis for a broad class of descent algorithms in nonsmooth optimization. We also show that this notion is instrumental for the design and justification of a novel R-class of algorithms to solve inclusion problems.
\end{abstract}



\section{Introduction}\label{intro}

A large number of problems of physical phenomena as well as in applied sciences (see, e.g.,\cite{abs,av,BC,Bento,br,br1,DonRoc09,L1,LT,Mordukhovich,Mordukhovich24,Nesterov3,Rockafellar,roc-wets,tbp} and the references therein) are described by inclusions of the type
\beq\label{main}
0\in \mathcal{A}(x),
\eeq 
where $\mathcal{A}:\mathbb{X} \rightrightarrows \mathbb{Y}$ is a set-valued mapping/multifunction between two Banach spaces $\mathbb{X}$ and $ \mathbb{Y}$. In order to find exact or approximate solutions to such inclusions, the qualitative properties of operators ${\cal A}$ are investigated, which are instrumental for the design and justification of numerical algorithms. A crucial issue in the study of \eqref{main} is related to stability of solutions in one or another sense. To address this issue, powerful tools of variational analysis have been developed in recent years. Among these tools, particular attention has been given to the properties of metric regularity and subregularity, various Lipschitz-type behaviors, etc  {(see, e. g., \cite{Adly,IOFFE,sacha2016,Lopez,NT})}. To proceed in this line, the notion of  {\it R-continuity} (Definition~\ref{rdef}) has been recently introduced  and developed in \cite{L1}, where its importance has been demonstrated in the study of stability of solutions to  (\ref{main}) and conducting convergence analysis of numerical algorithms to solve such inclusions. This concept generalizes the Lipschitz continuity property developed by Rockafellar  \cite{Rockafellar} when the solution set is a  singleton  and the modulus function is Lipschitz continuous. It has been realized that R-continuity is not as strong  as metric regularity being fulfilled for  a significantly broader class of operators. In particular, if ${\cal A}$ has closed graph at some point in its domain and is locally compact at this point, then the operator is R-continuous at the point in question; see \cite{L1,LT}. 

On the other hand, R-continuity is not too weak as calmness while being sufficient to ensure the {\it consistency}  of inclusion (\ref{main}) in the sense of  Hoffman \cite{Hoffman} meaning that a solution to the corresponding approximate problem is always near a solution to the original one, i.e., if $\Vert  y_\sigma \Vert$  is small enough and $y_\sigma\in \mathcal{A}(x_\sigma)$ for some $x_\sigma$, then there exists  $\bar x$ such that  $0\in \mathcal{A}(\bar x)$ and $\Vert x_\sigma-\bar x\Vert$ is  also small. The consistency in  Hoffman's sense  can be considered as a kind of {\it stability} and has a far-going impact to the development of sensitivity analysis in optimization and related areas.  In this sense, calmness is not sufficient to guarantee such a stability of (\ref{main}), while 
R-continuity is; see Section~\ref{s2}. 

The aforementioned Hoffman's consistency/stability is important in optimization as illustrated by the following discussion. Consider the minimization of an extended-real-valued, proper, convex, and lower semicontinuous function $f: \R^n \to\overline{\R}:=\R\cup \{\infty\}$. In this case, the minimization problem $\min_{x\in \R^n} f(x)$ is equivalently described by inclusion (\ref{main}), where  $\mathcal{A}=\partial f$ is the  subdifferential of $f$ in the sense of convex analysis, and where $S:=(\partial f )^{-1}(0)$ is the solution set to \eqref{main}. In practice,  due to  errors of  measurement, we can only find some $x_\sigma, y_\sigma \in \R^n$ such that $y_\sigma\in \partial f(x_\sigma)$ and $\Vert  y_\sigma \Vert$ is sufficiently small. Nevertheless, the R-continuity of $ (\partial f) ^{-1}$  at zero allows us to confirm that $x_\sigma$ is an approximate solution to \eqref{main} since there exists $\bar x\in S$ such that $x_\sigma$ is sufficiently close to $\bar x$. Consequently, the value of {$f( x_\sigma)$} approximates the optimal value $f(\bar x)$.  Conversely, the failure of the R-continuity of $(\partial f )^{-1}$ at zero tells us that $x_\sigma$ may be too far from the solution set $S$ (see Remark~\ref{rm1} for more details), and hence {$f( x_\sigma)$} is far from the optimal value.\vspace*{0.05in} 
 
The current paper shows that we can do even better in this direction. Based on the idea of R-continuity, the new notion of {\it compact R-continuity} is introduced (Definition~\ref{crdef}), which constitutes a weaker property compared to R-continuity that actually reduces to R-continuity restricted to compact sets. In this way, compact 
R-continuity does not require local compactness as R-continuity and thus broadly extends the spectrum of systems \eqref{main} that exhibit Hoffman's stability. This comes from the fact that in many reasonable settings, the (full)
R-continuity is not required and it is sufficient to deal just with its compact counterpart. A simple albeit crucial example showing that compact R-continuity is strictly weaker than R-continuity is given in 
Remark~\ref{rm1}. Indeed, it is sufficient for an operator to have a closed graph at some point in its domain to enjoy the compact R-continuity at this point. The converse is true if the operator value at the reference point is a closed set; see Theorem \ref{compact}. A remarkable result established below (Theorem~\ref{loja}) tells us that the compact R-continuity of the solution mapping associated with the analytic finite-dimensional equation in \eqref{main} can be derived from the seminal {\it \L ojasiewicz inequality} in the semialgebraic  geometry \cite{Lojasiewicz}. We'll discuss possible extensions of this compact R-continuity result to more general settings. 

Important applications of compact R-continuity discovered below concern {\it global convergence} of a large class of {\it descent algorithms} in problems of nonsmooth and nonconvex optimization. It has been well recognized in numerical optimization that such a convergence analysis can be efficiently conducted under the {\it 
Polyak-\L ojasiewicz-Kurdyka} (PKL) conditions imposed on extended-real-valued cost functions $f$; see, e.g., \cite{ab,abs,Bento,karimi,kurduka,Lojasiewicz1,polyak} and the references therein. The conditions of the PLK type ensure the global convergence of descent algorithms to stationary points of optimization problems defined via various subdifferentials of the cost function $f$, i.e., zeros of $\nabla f$ in the case of smooth functions. We show in Theorem~\ref{conver} that the desired convergence is guaranteed when PLK conditions are replaced by the (compact) R-continuity of the stationary solution mapping $(\partial f)^{-1}$. The results of this novel type lead us to significant advantages for a broad class where the (compact) R-continuity is satisfied.
Interestingly, compact R-continuity can explain why the boundedness of generated sequences can ensure obtaining approximate solutions in many important algorithms.

Our approach to convergence analysis of algorithms applies not only to optimization-related problems but to solving general inclusions of type \eqref{main}, where ${\cal A}$ may not be a subdifferential operator. In such settings, we propose a generic algorithm of {\it R-class} and establish conditions that guarantee the global convergence of iterates. The imposed conditions (Theorem~\ref{converg}) involve R-continuity and compact R-continuity assumptions but not assumptions of the PLK type. Our theoretical developments can be applied to obtained new convergence results for solving inclusion \eqref{main} governed by maximally monotone operators or shifts of the inverses by using the Proximal Point Algorithm and to  DC (Difference of Convex) programming by using the DCA (Difference of Convex Algorithm), where the compact R-continuity assumption holds automatically.

The rest of paper is organized as follows. Section~\ref{s2} recalls the needed preliminaries from variational analysis, {then} defines and discusses the basic notions of our study. In Section~\ref{s3}, we establish some fundamental properties of compactly R-continuous multifunctions with deriving sufficient conditions for this property and its relationship with the \L ojasiewicz inequality for analytic functions. Section~\ref{s4} focuses on the global convergence of generic optimization algorithms, where the R-continuity and compact R-continuity properties play a crucial role. In Section~\ref{sec:inc}, we address solving general inclusions \eqref{main} and their maximally monotone specifications  with the design and justification of novel R-class algorithms. The final Section~\ref{s5} contains concluding remarks and discusses some directions of our future research. Everywhere in this paper, we use the standard notation and terminology of variational analysis; see, e.g., \cite{Mordukhovich,roc-wets}.
 
\section{Basic Definitions and Preliminaries}\label{s2}

Let  $\mathcal{A}: \mathbb{X} \rightrightarrows \mathbb{Y}$ be a set-valued mapping between Banach spaces, and let $\bar x$ be such that ${\cal A}(\bar x)\ne\emptyset$.  The {\it domain}, {\it range,} and  {\it graph} of $\mathcal{A}$ are defined, respectively, by 
$$
{\rm dom}\,\mathcal{A}:=\big\{x\in \mathbb{X}\;\big|\;\mathcal{A}(x)\neq \emptyset\big\},\;\;{\rm rge}\,\mathcal{A}:=\cup_{x\in \mathbb{X}}\mathcal{A}(x),\;\;{\rm gph \,}{\mathcal{A}:=\big\{(x,y)\in \mathbb{X}\times \mathbb{Y}\;\big|\;y\in\mathcal{A}(x})\big\}.
$$
As usual, denote by $\mathcal{A}^{-1} :\mathbb{Y}\rightrightarrows \mathbb{X}$ the {\it inverse} of $\mathcal{A}$ defined by 
$$
x\in \mathcal{A}^{-1}(y)\; \iff \;  y\in \mathcal{A}(x).
$$ 
We say that $\mathcal{A}: \mathbb{X} \rightrightarrows \mathbb{Y}$  is  \textit{metrically regular} around $(\bar{x},\bar{y})\in {\rm gph}(\mathcal{A})$  if there exists $\kappa>0$ such that
 \begin{equation}\label{metric}
\mathbf{d}(x,\mathcal{A}^{-1}(y))\le \kappa \mathbf{d}(y,\mathcal{A}(x))
\end{equation}
for all $(x,y)$ near $(\bar{x},\bar{y})$, where $\mathbf{d}(x,\Omega):=\inf_{y\in \Omega}\Vert x-y\Vert$ denotes the {\it distance function} from $x$ to the  set $\Omega$.  The multifunction $\mathcal{A}$ is \textit{metrically subregular} at $(\bar{x},\bar{y})\in {\rm gph}\,\mathcal{A}$ if $y=\bar y$ in (\ref{metric}), i.e.,
\begin{equation*}
\mathbf{d}(x,\mathcal{A}^{-1}(\bar{y}))\le \kappa \mathbf{d}(\bar{y},\mathcal{A}(x))
\end{equation*} 
for all $x$ near $\bar{x}$. It is said that $\mathcal{A}: \mathbb{X} \rightrightarrows \mathbb{Y}$ is {\it Lipschitz-like} (or has the {\it Aubin property}) around $(\bar{x},\bar{y})\in {\rm gph}\,\mathcal{A}$ if  there exist $\kappa>0$ and some neighborhoods $U$ of $\bar{x}$ and $V$ of $\bar{y}$ such that 
\beq\label{aubin}
\mathbf{d}(y_1,\mathcal{A}(x_2))\le \kappa \mathbf{d}(x_1,x_2) \;{\rm for \; all\;} y_1\in \mathcal{A}(x_1) \cap V \;{\rm and \;} x_1, x_2\in U.
\eeq 
We say that $\mathcal{A}$ is {\it calm} at $(\bar{x},\bar{y})$  if $x_2=\bar{x}$ in (\ref{aubin}), i.e., 
\beq\label{calm}
\mathbf{d}(y_1,\mathcal{A}(\bar{x}))\le \kappa \mathbf{d}(x_1,\bar{x}) \;{\rm for \; all\;} y_1\in \mathcal{A}(x_1) \cap V \;{\rm and\;} x_1\in U.
\eeq 

It has been well recognized in variational analysis that $\mathcal{A}$ is metrically regular around $(\bar{x},\bar{y})\in {\rm gph}\,\mathcal{A}$ if and only if and only if the inverse mapping $\mathcal{A}^{-1}$ is 
Lipschitz-like around $(\bar{y},\bar{x})$. Similarly, the metric subregularity of $\mathcal{A}$ at $(\bar{x},\bar{y})\in {\rm gph}(\mathcal{A})$ is equivalent to the calmness of $\mathcal{A}^{-1}$ at $(\bar{y},\bar{x})$ {(see, e.g., \cite{IOFFE})}. However, metric regularity and Lipschitz-like property may be too strong, since they require the corresponding estimate to hold in all the points of a neighborhood of the reference one that is not always needed.  On the other hand, metric subregularity and calmness are rather weak to provide useful information on the desired stability. This is the case when  $\mathcal{A}(x_1) \cap V$ is an empty set in (\ref{aubin}); see \cite{LT}) for more discussions. Moreover, even if the set  $\mathcal{A}(x_1) \cap V$  is nonempty, there exists some element $y\in \mathcal{A}({x_1})\setminus V$ not satisfying (\ref{calm}), i.e., it can be far enough from the solution set of $\mathcal{A}(\bar x)$. This was a motivation in \cite{L1,LT} to design something in the middle, which still be able nevertheless to serve our initial goal discussed in Section~\ref{intro}. 

\begin{definition}\label{rdef}
The set-valued mapping \noindent $\mathcal{A}:\mathbb{X} \rightrightarrows \mathbb{Y}$ is  called {\sc R-continuous} at ${\bar x}  $ if there exist $\sigma>0$ and a nondecreasing function $\rho: \mathbb{R}^+\to \mathbb{R}^+$ satisfying $\lim_{r\to 0^+}\rho(r)=\rho(0)=0$  such that
 \begin{equation}\label{rcon}
\mathcal{A}(x) \subset \mathcal{A}({\bar x}  )+\rho(\Vert x-{\bar x}   \Vert)\ball\;\mbox{ for all }\;x\in  \ball({\bar x} ,\sigma)
\end{equation}
with the {\sc continuity modulus function} $\rho$ and {\sc radius} $\sigma$. When $\sigma=\infty$, $\mathcal{A}$ is  called {\sc globally R-continuous} at ${\bar x}$. The inclusion in \eqref{rcon} means that for each $y\in \mathcal{A}(x)$, there exists ${\bar y}  \in  \mathcal{A}({\bar x} )$ satisfying the estimate $\Vert y-{\bar y}  \Vert\le \rho(\Vert x-{\bar x}   \Vert)$ for all $x\in \ball({\bar x},\sigma)$. We say that 
\begin{itemize}
\item $\mathcal{A}$ is {\sc $R$-Lipschitz continuous} at ${\bar x}  $ with modulus $L>$  if $\rho( r)=Lr$.
\item $\mathcal{A}$ is {\sc $R$-H\"older continuous} at ${\bar x}  $  if $\rho( r)=Lr^\theta$ for some $L>0, \theta>0$.
\end{itemize}
\end{definition}\vspace*{0.05in}

Note that the R-continuity in \eqref{rcon} is obviously equivalent to
\begin{equation*}
e(\mathcal{A}(x), \mathcal{A}({\bar x}))\le\rho(\Vert x-{\bar x}   \Vert)\;\mbox{ for all }\;x\in \ball({\bar x}  ,\sigma),
\end{equation*}
where ${e}(A,B):=\sup_{x\in A}\mathbf{d}(x,B)$  signifies the {\it excess} of the set $A$ over $B$. In this terms, we see that:

$\bullet$ $\mathcal{A}$ is {\it $R$-Lipschitz continuous} at ${\bar x}  $ with modulus $L>0$ if and only if there is $\sigma>0$ such that 
\begin{equation*}
\sup_{x\in \ball({\bar x} ,\sigma)}\frac{e(\mathcal{A}(x), \mathcal{A}({\bar x}))}{\Vert x-{\bar x}   \Vert}\le L
\end{equation*}
with the convention $\frac{0}{0}:=0$.

$\bullet$ $\mathcal{A}$ is {\it $R$-H\"older continuous} at ${\bar x}  $ if and only if there are numbers $\sigma, L,\theta>0$ such that 
\begin{equation*}
\sup_{x\in \ball({\bar x}  ,\sigma)}\frac{e(\mathcal{A}(x), \mathcal{A}({\bar x}))}{\Vert x-{\bar x}   \Vert^\theta}\le L.
\end{equation*}\vspace*{0.05in}

Next we introduce the notion of {\it compact R-continuity} of ${\cal A}$, which is actually the R-continuity of this mapping restricted on compact subsets of $\mathbb X$. 

\begin{definition}\label{crdef}
A set-valued mapping \noindent $\mathcal{A}:\mathbb{X} \rightrightarrows \mathbb{Y}$ is  called {\sc compactly 
R-continuous} at ${\bar x} $ if for every compact set $K\subset \mathbb{X}$, there exists $\sigma>0$ such that we can find a nondecreasing function $\rho: \mathbb{R}^+\to \mathbb{R}^+$ satisfying $\lim_{r\to 0^+}\rho(r)=\rho(0)=0$ with
\begin{equation}\label{rconc}
\mathcal{A}(x)\cap K \subset \mathcal{A}({\bar x}  )+\rho(\Vert x-{\bar x}   \Vert)\ball\;\mbox{ for all }\;x\in  \ball({\bar x} ,\sigma).
\end{equation}
Similarly to Definition~{\rm\ref{rdef}}, we say that 
\begin{itemize}
\item $\mathcal{A}$ is {\sc compactly $R$-Lipschitz continuous} at ${\bar x} $ with modulus $L>0$ if $\rho( r)=Lr$.

\item $\mathcal{A}$ is {\sc compactly $R$-H\"older continuous} at ${\bar x}$ if $\rho( r)=Lr^\theta$ for some 
$L, \theta>0$.
\end{itemize}
\end{definition}\vspace*{0.05in}

As above, we have the equivalent descriptions of the notions from Definition~\ref{crdef} via the excess: 
\begin{itemize}
\item $\mathcal{A}$ is compactly $R$-Lipschitz continuous at ${\bar x}  $ with modulus $L>0$ if and only if for every compact set $K\subset \mathbb{X}$, there exists $\sigma>0$ such that
$$
\sup_{x\in \ball({\bar x}  ,\sigma)}\frac{e(\mathcal{A}(x)\cap K, \mathcal{A}({\bar x}))}{\Vert x-{\bar x}   \Vert}\le L.
$$
\item $\mathcal{A}$ is compactly $R$-H\"older continuous at ${\bar x}$  if and only if for every compact set $K\subset \mathbb{X}$, there exist constants $\sigma,L,\theta>0$ such that 
$$
\sup_{x\in \ball({\bar x} ,\sigma)}\frac{e(\mathcal{A}(x)\cap K, \mathcal{A}({\bar x}))}{\Vert x-{\bar x}   \Vert^\theta}\le L.
$$
\end{itemize}

\begin{remark}\label{rm1} {\rm 
The notion of  compact R-continuity is  strictly weaker than that of R-continuity.  If $\mathcal{A}$ is R-continuous at ${\bar x}$, then it is obviously compactly R-continuous at this point. However, the converse implications fails in general. Indeed, taking $\mathbb{X}=\mathbb{Y}=\R$, $A(0)=0$, and $A(x)= \{x,\frac{1}{x}\}$ for $x\neq 0$, we see that $A$ is compactly R-Lipschitz continuous but not R-continuous.} 
\end{remark}

The following proposition comes directly from the above definitions.

\begin{proposition}\label{localc}
If $\mathcal{A}$ is locally compact around  ${\bar x}$, i.e.,  there is $\sigma>0$ such that $\mathcal{A}( \ball({\bar x}  ,\sigma))$ belongs to a compact set, then the compact R-continuity of ${\cal A}$ at ${\bar x}$ yields the R-continuity ${\cal A}$ at  this point.
\end{proposition}

Observe that when $\mathbb{Y}=\mathbb{R}^m$, the local compactness of the multifunction  $\mathcal{A}$ around ${\bar x}\in\dom\,{\cal A}$ is equivalent to the {\it local boundedness} of ${\cal A}$ around this point, i.e., the existence of $\sigma>0$ such that the set $\mathcal{A}(\ball({\bar x} ,\sigma))$ is  bounded in $\R^m$.

\section{Fundamental Properties of Compact R-continuity}\label{s3}

In this section, we establish some important properties of compactly R-continuous multifunctions that are crucial for their subsequent applications. The next definition formulates the property of general set-valued mappings that is useful in what follows. 

\begin{definition}\label{closed graph} Given a set-valued mapping $\mathcal{A}:\mathbb{X} \rightrightarrows \mathbb{Y}$, we say that the {\sc graph} of ${\cal A}$ is $($sequentially$)$  {\sc closed at a domain point} $\bar x\in\dom\,{\cal A}$ if for any sequences $x_k\to\bar x$ and any sequence $y_k\in{\cal A}(x_k)$ converging to some $y$ as $k\to\infty$, we have $y\in{\cal A}(\bar x)$.
\end{definition}

The following theorem provides a sequential characterization of the compact R-continuity of set-valued mappings. Without loss of generality, we suppose that $\bar x=0$.

\begin{theorem}\label{compact}
A set-valued mapping $\mathcal{A}:\mathbb{X}\rightrightarrows \mathbb{Y}$ is compactly R-continuous at zero {and $\mathcal{A}(0)$ is closed} if and only if its graph is closed at this point.
\end{theorem}
\begin{proof}  
To deduce the compact R-continuity at zero from the imposed closed graph property, observe first the latter obviously ensures that the set $\mathcal{A}(0)$ is closed. Given a compact set $K\subset \mathbb{X}$, define the function $\rho: \mathbb{R}^+\to \mathbb{R}^+$ by
\beq\label{defro}
\rho(\sigma):=\inf\big\{\delta>0\;\big|\;\mathcal{A}(x) \cap K \subset \mathcal{A}(0)+\delta\ball\;\mbox{ for all }\;x\in \sigma \ball\},\quad \sigma\ge 0.
\eeq
It is easy to see that $\rho$ is well-defined and nondecreasing with $\rho(0)=0$.  It follows from the nondecreasing property of $\rho$ and its boundedness from below that the limit $\lim_{\sigma\to 0^+}\rho(\sigma)$ exists. Suppose that  $\lim_{\sigma\to 0^+}\rho(\sigma):=\overline{\delta}>0$. Then there exist two sequences $(x_k)$, $(y_k)$ such that $x_k \to 0$, $y_k\in \mathcal{A}(x_k)\cap K$,  and we have
\beq\label{closedr}
y_k\notin \mathcal{A}(0)+\frac{\overline\delta}{2}\ball.
\eeq
Since the set  $K$ is compact and the graph of $\mathcal{A}$ is closed at zero, it follows without relabeling that the sequence $(y_k)$ converges to some  $\bar y \in\mathcal{A}(0)$ as $k\to\infty$. This clearly contradicts  the condition in $(\ref{closedr})$ and hence shows that  $\lim_{\sigma\to 0^+}\rho(\sigma)=0$. Picking now $x\in \mathbb{X}$ with $\Vert x \Vert=\sigma$, we deduce from (\ref{defro}) and the closedness of $\mathcal{A}(0)$ that
$$
\mathcal{A}(x) \cap K \subset \mathcal{A}(0)+\rho(\sigma)\ball = \mathcal{A}(0)+\rho(\Vert x \Vert)\ball, 
$$
which readily justifies the compact R-continuity property of $\mathcal{A}$ at zero.

To verify the reverse implication, take $x_k\to 0$ and $y_k\to y$ as $k\to\infty$ with $y_k\in \mathcal{A}(x_k)$ and then define the set $K:=\{ y_k\;|\;k\in\N\}$, which is obviously compact in $Y$. It follows from the 
compact R-continuity of $\mathcal{A}$  at zero that
$$
y_k\in \mathcal{A}(x_k) \cap K\subset \mathcal{A}(0)+\rho(\Vert x_k\Vert)\ball\;\mbox{ for all }\;k=1,2,\ldots.
$$
Since $y_k\to y$, $\rho(\Vert x_k \Vert)\to 0$ and the set $\mathcal{A}(0)$ is closed, we get that  $y\in \mathcal{A}(0)$, and thus the graph of $\mathcal{A}$ is closed at zero. This completes the proof of the theorem.
\end{proof}

\begin{corollary}\label{corrc}
If the mapping $\mathcal{A}:\mathbb{X}\rightrightarrows \mathbb{Y}$ is locally compact around zero and its graph is closed {at zero}, then ${\cal A}$ is R-continuous at this point.
\end{corollary}
\begin{proof}
This follows directly from Theorem~\ref{compact} and Proposition~\ref{localc}.
\end{proof}

The next result provides a simple condition allowing us to deduce the R-Lipschitz continuity at zero from the 
R-continuity at zero for a class of inverse mappings in finite dimensions.

\begin{theorem}\label{rconti} Let $f:\R^n\to\R^m$ with $m\ge n$ be continuously differentiable and the Jacobian matrix $\nabla f$ have full rank on the solution set $S:=f^{-1}(0)$. If the inverse mapping $f^{-1}$ is R-continuous at zero, then it is R-Lipschitz continuous at this point.
\end{theorem}
\begin{proof} Assuming that $f^{-1}$ is R-continuous at zero with radius $\sigma>0$, take $y\in \R^n$ with $\Vert y\Vert\le \sigma$ and $x\in f^{-1}(y)$. The R-continuity of $f^{-1}$ at zero gives us some $u\in S$ near $x$. Using  the Taylor expansion around $u$ (the radius can be adjusted if necessary), we can write
$$
y=f(x)=f(u)+\nabla f(u)(x-u)+ o(\Vert x-u\Vert)=\nabla f(u)(x-u)+ o(\Vert x-u\Vert).
$$
Since $\nabla f(u)$ is of full rank and $\Vert x-u\Vert$ is small, there is $c>0$ (see, e.g,, \cite{L1,tbp}) such that
$$
\Vert y\Vert=\Vert \nabla f(u)(x-u)+ o(\Vert x-u\Vert)\Vert\ge \Vert\nabla f(u)(x-u)\Vert- o(\Vert x-u\Vert) \ge c\Vert  x-u\Vert.
$$
This readily leads us to the inclusion
$$
f^{-1}(y) \subset  f^{-1}(0)+\frac{\Vert y\Vert}{c} \ball\;\mbox{ for all }\;y\in \ball(0,\sigma),
$$
which therefore verifies the claimed assertion.
\end{proof}

\begin{remark} {\rm Theorem~\ref{rconti} is an extension of \cite[Proposition~4]{Rockafellar} in finite-dimensional spaces when $S=\{\bar x\}$ is singleton which requires that $f: \R^n\to \R^n$ is single-valued and $C^1$ in a neighborhood of $\bar x$ with $\nabla f(\bar x)$ being invertible. Then $f^{-1}$ is differentiable at zero which implies that $f^{-1}$ is of closed graph  at zero and locally bounded around this point zero. Thus $f^{-1}$ is R-continuous at zero (Corollary ~\ref{corrc}) and then Lipschitz R-continuous therein.}
\end{remark}

The following theorem tells us that the general assumptions of Theorem~\ref{rconti} unconditionally guarantee that the inverse mapping $f^{-1}$ is compactly R-Lipschitz continuous at zero.

\begin{theorem}\label{cr-inv} Let $f:\R^n\to\R^m$ with $m\ge n$ be continuously differentiable and $\nabla f$ have full rank on the solution set $S:=f^{-1}(0)$. Then $f^{-1}$ is compactly R-Lipschitz continuous at zero.
\end{theorem}
\begin{proof}
Since $f$ is continuous, the graphs of $f$ and hence of $f^{-1}$ are closed. Therefore, Theorem~\ref{compact} implies that $f^{-1}$ is compactly R-continuous at zero. Given further a compact set $K$, we can find the corresponding radius $\sigma>0$ ensuring the estimate in \eqref{rconc}. Take further $y\in \R^n$ with $\Vert y\Vert\le \sigma$ and pick $x\in f^{-1}(y) \cap K$. From the compact R-continuity of $f^{-1}$ at zero, we deduce the existence of  the desired element $u\in S$ near $x$ and then proceed similarly to the proof of Theorem~\ref{rconti}.
\end{proof}

Now we relate the compact R-continuity of the solution mapping $S=f^{-1}$ at zero (in fact, its stronger {\it H\"olderian} version) with the classical {\it \L ojasiewicz inequality} for the case of analytic functions in  finite-dimensional spaces. The result below can be found in \cite{Lojasiewicz,Lojasiewicz1}, {see also \cite{dp1,dp2}}. 

\begin{theorem}\label{Lojasiewicz}
Let $f\colon U\to\R$ be an analytic function defined on an open set $U\subset\R^n$.  Let $S:=f^{-1}(0)$ and suppose that $S\neq \emptyset$.
Then for every compact set $K\subset U$, there exist $\theta,c>0$ such that 
\begin{equation}\label{loj-ine}
[\mathbf{d}(x,S)]^\theta\le c \vert f(x)\vert\;\mbox{ for all }\;x\in K.
\end{equation}
\end{theorem}

The next theorem shows that the \L ojasiewicz inequality \eqref{loj-ine} yields the fulfillment of the compact H\"older R-continuity of the solution mapping $f^{-1}$ at zero generated by analytic functions. As illustrated by Example~\ref{ex1} below, the reverse implication fails in general. 

\begin{theorem}\label{loja}
Let $f\colon U\to\R$ be an analytic function defined on an open set $U\subset\R^n$, and let $S$ be the  zeros of $f$, i.e., $S=f^{-1}(0)$. If $S\ne\emptyset$, then $f^{-1}$ is compactly R-H\"older continuous at zero.
\end{theorem}
\begin{proof}
Given a compact set $K\subset U$, we use Theorem \ref{Lojasiewicz} and find positive constants $\theta$ and $c$ such
that the \L ojasiewicz inequality \eqref{loj-ine} is satisfied. Take $y$ of a  sufficiently small absolute value  and select $x\in f^{-1}(y)\cap K$. Then we have $y=f(x)$ and $x\in K$ satisfying the implication
 $$
\big[\mathbf{d}(x,S)]^\theta\le c \vert f(x)\vert = c \vert y\vert\big]\Longrightarrow \big[\mathbf{d}(x,S)\le c^{1/\theta} \vert y\vert^{1/\theta}\big],
$$
which immediately yields the inclusion
$$
x\in S+c^{1/\theta} \vert y\vert^{1/\theta}\ball.
$$
Since the latter is true for all $x\in f^{-1}(y)\cap K$, we arrive at
$$
f^{-1}(y)\cap K\subset S+c^{1/\theta} \vert y\vert^{1/\theta}\ball,
$$
and thus conclude the proof of the theorem.
\end{proof}

Here is the aforementioned example on the advantage of compact R-continuity.

\begin{example}\label{ex1}
{\rm Consider the function $f\colon\R\to\R$ defined by
$$
f(x):= \left\{
\begin{array}{l}
e^{-\frac{1}{ x^2}} \;\;\;\;{\rm if}\;\; \;\;x\neq 0,\\
0 \;\;\;\;\;\;\;\;\;\;{\rm if}\;\;\;\;\;x=0.
\end{array}\right.
$$
This function is of class $C^\infty$ with the solution set $S=\{0\}$, while it satisfies neither analyticity nor the \L ojasiewicz inequality; see, e.g., \cite{Bdlm}). Nevertheless, Corollary~\ref{corrc} tells us that the inverse mapping $f^{-1}$ is R-continuous at zero since its graph is closed and $f^{-1}$ is locally bounded around zero}. 
\end{example}

Let us discuss further relationships between the versions of R-continuity under consideration and \L ojasiewicz-type estimates.

\begin{remark}$\,$ {\rm  We make the following commentaries to the obtained result and their extensions:

{\bf(i)} \L ojasiewicz' inequality \eqref{loj-ine} concerns the target function $f$ directly while
R-continuity acts on the inverse of $f$. Consequently, \L ojasiewicz's  inequality deals with the collections of {\it inputs} $x$ while R-continuity applies to the set of {\it outputs} $y$. The set of all inputs is usually very large. On the contrary, starting with outputs leads us to checking much smaller input sets. Indeed, for the equation $y_0=f(x)$, we may get only an approximation $y_\epsilon$ of $y_0$ and the corresponding solution $x_\epsilon$. In practice, it is often observed that a small amount of $\Vert y_{\epsilon_1}-y_{\epsilon_2}\Vert$ deduces a small one for $\Vert x_{\epsilon_1}-x_{\epsilon_2}\Vert$, which is actually the spirit of R-continuity. 

{\bf(ii)} While \L ojasiewicz's inequality from Theorem~\ref{Lojasiewicz} is restricted by analytic functions in finite dimensions, R-continuity and its compact version are much broader giving us sufficient information to verify Hoffman's consistency/stability for inclusion \eqref{main} as discussed in Section~\ref{intro}.

{\bf(iii)} There are various extensions of the \L ojasiewicz inequality \eqref{loj-ine}, which go far beyond analytic functions $f$ and sets $S=f^{-1}(0)$; see the references in Section~\ref{intro}. In particular, the paper by Polyak \cite{polyak} concerns continuously differentiable functions with Lipschitz continuous derivatives in Hilbert spaces and derives a gradient estimate called now the Polyak-\L ojasiewicz condition \cite{Bento,karimi}. Other conditions in this line but of a different type are known as Kurdyka-\L ojasiewicz properties \cite{ab,abs,Bdlm} being applied to the solution mappings in the subdifferential form $S=(\partial f)^{-1}$. All the Polyak-\L ojasiewicz-Kurdyka conditions mentioned above are motivated by applications to numerical algorithms of optimization; see Section~\ref{s4} for more details. It is appealing to clarify whether such conditions can be used to establish {\it stability of subgradient inclusions} $0\in\partial f(x)$ in \eqref{main} via R-continuity as in Theorem~\ref{loja} for analytic equations. This will be a topic of our future research.}
\end{remark}

\section{Applications to Convergence of Optimization Algorithms} \label{s4}

The main attention of this section is to present applications of (compact) R-continuity  for subgradient inclusions with ${\cal A}(x)=\partial f(x)$ to the global convergence of a broad class of descent algorithms in nonsmooth constrained optimization. We consider here a general class of minimized cost functions $f\colon\R^n\to\overline{\R}$, which are extended-real-valued and lower semicontinuous (l.s.c.).  It has been well recognized that the PLK type conditions play a crucial role in establishing global convergence of descent minimization algorithms. For completeness and reader's convenience, we recall the formulation of the major PLK conditions, which contain the  subdifferential constructions used below to describe the corresponding algorithms and stationary points. The following definition is taken from \cite{Bento}, where the reader can find more references and discussions.

\begin{definition}\label{def:kl} Let $f:\mathbb{R}^n \to\overline{\mathbb R}$ be a proper extended-real-valued l.s.c.\ function with the domain  ${\rm dom}\,f:=\{x\in\mathbb{R}^n\;|\;f(x)<\infty\}$. We say that the function $f$ satisfies:

{\bf(i)} The {\sc PLK condition} at $\Bar{x}\in{\rm dom}\,f$ if there are $\eta \in (0,\infty)$, a neighborhood $U$ of $\Bar{x}$, and a concave continuous function $\varphi: [0,\eta] \to[0,\infty)$, called the {\sc desingularizing function}, such that
\begin{equation}\label{kl0}
\varphi(0)=0,\quad\varphi\in C^1(0,\eta),\quad\varphi'(s)>0\;\mbox{ for all }\; s\in(0,\eta),\;\mbox{ and }\;
\end{equation}
\begin{equation}\label{desigualdadeklND}
\varphi'\big(f(x)-f(\Bar{x})\big)\mathbf{d}\big(0,\partial f(x)\big)\ge 1\;\mbox{ for all }\;x\in U\cap[f (\Bar x) <f(x)<f(\Bar{x})+\eta],
\end{equation}
where $\partial f(\Bar x)$ stands for the {\sc Mordukhovich/limiting subdifferential} of  $f$ at $\Bar{x}$ defined by
\begin{equation}\label{ls}
\partial f(\Bar{x}):=\Big\{v\in\mathbb{R}^n\;\Big|\;\exists x_k\stackrel{f}{\to}\Bar{x}, \,v_k\in\widehat\partial f(x_k),\quad v_k\to v\;\mbox{ as }\;k\to\infty\Big\}
\end{equation}
with $x_k\stackrel{f}{\to}\Bar{x}$ meaning that $x_k\to \Bar{x}$, $f(x_k)\to f(x)$ and with $v_k\in\widehat\partial f(x_k)$ meaning that 
$$
{\limsup}_{u\to x_k}\frac{f(u)-f(x_k)-\langle v_k,u-x_k\rangle}{\|u-x_k\|}\ge 0.
$$

{\bf(ii)} The {\sc symmetric PLK condition} at $\Bar{x}$ if $f$ is continuous around $\Bar{x}$ and $\partial f$ is replaced in \eqref{desigualdadeklND} by the {\sc symmetric subdifferential} of $f$ at $\Bar{x}$ defined by
\begin{equation}\label{sym}
\partial^0f(\Bar{x}):=\partial f(\Bar{x})\cup\big(-\partial(-f)(\Bar{x})\big).
\end{equation}

{\bf(iii)} The {\sc strong PLK condition} if $f$ is Lipschitz continuous around $\Bar{x}$ and $\partial$ is replaced in \eqref{desigualdadeklND} by the {\sc Clarke/convexified subdifferential} of $f$ at $\Bar{x}$ given by
\begin{equation}\label{cl}
\overline{\partial}f(\Bar{x}):={\rm co}\,\partial f(\Bar{x}),
\end{equation}
where ``{\rm co}" stands for the convex hull of the set in question.

{\bf(iv)} The {\sc exponent} versions of the {\sc PLK conditions} in {\rm(i)--(iii)} if the desingularizing function in \eqref{kl0}  and \eqref{desigualdadeklND} is selected as $\varphi(t)=M t^{1-q}$, where $M$ is a positive constant, and where $q\in[0,1)$. We refer to the case where $q\in(0,1/2)$ as the {\sc PLK conditions with lower exponents}.
\end{definition}

In \cite{abs}, a generic class of abstract descent methods was introduced for the minimization of l.s.c.\ functions $f$ satisfying the following properties:\\[1ex]
\noindent ${\bf(H_1)}$ {\it Sufficient decrease}: there exists some $\alpha>0$ such that for each $k=0,1,\ldots$ we have 
$$
f(x_k)-f(x_{k+1})\ge \alpha \Vert x_{k+1}-x_{k}\Vert^2.
$$
${\bf(H_2})$ {\it Relative error}: there exists some $\beta >0$ such that for each $k=0,1,\ldots$ we have:
$$
\exists \;w_{k+1}\in \partial f (x_{k+1})\;\;{\rm with}\;\; \Vert w_{k+1} \Vert \le \beta \Vert x_{k+1}-x_{k}\Vert.
$$

As shown in \cite{abs}, great many efficient optimization algorithms satisfy ${\bf(H_1)}$ and ${\bf(H_2)}$. These properties hold, in particular, for the {\it Proximal Point Algorithm} (PPA):
\beq\label{opcopr}
x_{k+1}\in {\rm Argmin }_{x\in \R^n}\{f(x)+\alpha\Vert x-x_{k}\Vert^2\}
\end{equation}
for which the necessary optimality (stationarity) condition reads as
\begin{equation*}
0\in\partial f(x_{k+1})+2\alpha (x_{k+1}-x_{k})
\end{equation*}
while ${\bf(H_2)}$ reduces to $w_{k+1}=-2\alpha (x_{k+1}-x_{k})$. On the other hand, there are important descent algorithms, which satisfy ${\bf(H_1)}$ but not ${\bf(H_2)}$. In particular, it happens for the {\it Boosted Difference of Convex Algorithm} ({BDCA) to solve problems of DC (difference of convex) programming proposed in \cite{av} for Lipschitzian cost functions, where it was shown that BDCA satisfied the modified property ${\bf(H_2)}$  with the replacement of $\Vert w_{k+1} \Vert \le \beta \Vert x_{k+1}-x_{k}\Vert$ by $\Vert w_{k} \Vert \le \beta \Vert x_{k+1}-x_{k}\Vert$ while taking some subgradient $w_{k}\in \overline{\partial}f (x_{k})$ from the convexified subdifferential \eqref{cl}.  Motivated by this, the following limiting subdifferential \eqref{ls} modification of ${\bf(H_2)}$  was proposed and developed in \cite{Bento}:\\[1ex]
\noindent${\bf(H_3)}$ {\it Modified relative error}: for each $k=0,1,\ldots$, there exists
\begin{equation*}
w^{k}\in\partial\phi(x^{k})\;\mbox{ with }\;\|w^{k}\|\le\beta\|x^{k+1} - x^k\|.
\end{equation*}

Based on the basic PLK condition, the global convergence of iterates to a stationary point from $(\partial f)^{-1}(0)$ was established in \cite{abs}  for the general algorithm under ${\bf(H_1)}$, ${\bf(H_2)}$ and in \cite{Bento} for the generic algorithm  under ${\bf(H_1)}$, ${\bf(H_3)}$ while assuming in addition that in both cases we have the property:\\[1ex]
\noindent ${\bf(H_4)}$ {\it Continuity condition}:
there exist a subsequence $\{x^{k_j}\}$ and $\bar x\in{\rm dom}\,f$ such that
\begin{equation*}
x^{k_j}\to{\bar x}\;\text{ and }\;f(x^{k_j}) \longrightarrow f(\bar{x})\;\mbox{ as }\;j\to\infty.
\end{equation*}
The symmetric PLK condition using \eqref{sym} was employed in \cite{Bento} to establish the global convergence for BDCA with continuous cost functions, which improved \cite{av} even for problems with Lipschitzian costs where \eqref{cl} was {used}. On the other hand, {\it convergence rates} obtained in \cite{Bento} under ${\bf(H_2)}$ and under  ${\bf(H_3)}$ in addition to  ${\bf(H_1)}$, ${\bf(H_4)}$, and the exponent PLK conditions are very different in the case of lower exponents; see \cite{Bento} for more details.\vspace*{0.05in}

The discussions above reflect the {\it state-of-the-art} in convergence analysis of the generic algorithms under consideration. These discussions are presented to emphasize the difference between the PLK-based approach and the following theorem involving {\it R-continuity} and its {\it compact} version. Similarly to Theorem~\ref{Lojasiewicz}, we expect advantages of the new R-continuity approach to find approximate solutions of optimization problems via the generic classes of algorithms discussed above in comparison with the conventional approach to global convergence based on the PLK conditions. Observe that the theorem below does not require the fulfillment of ${\bf(H_4)}$.

\begin{theorem}\label{conver}
Let be given an l.s.c. $f : \R^n \to\overline{\R}$ with $\inf_{x\in \R^n} f(x) >-\infty$, and let $(x_k)$ be a sequence of iterates generated by a descent method satisfying ${\bf(H_1)}$ and either ${\bf(H_2)}$, or ${\bf(H_3)}$.  
Assuming in addition that the set of stationary points $S:=(\partial f)^{-1}(0)$ is nonempty, we have: 

{\bf(i)} If $(\partial f)^{-1}$ is R-continuous at zero, then $ \mathbf{d}(x_k,S)\to 0$ as $k\to \infty$.

{\bf(ii)} If  the sequence of iterates $(x_k)$ is bounded and the graph of {the inverse of}  limiting subgradient mapping \eqref{ls} is closed at zero,  then $\mathbf{d}(x_k,S)\to 0$ as $k\to \infty$.
\end{theorem}
\begin{proof}
It follows directly from ${\bf(H_1)}$ that 
$$
\sum_{k=0}^\infty\Vert x_{k+1}-x_{k}\Vert^2\le f(x_0)-\inf_{x\in \R^n} f(x)<\infty, 
$$
which readily implies that  $\Vert x_{k+1}-x_{k}\Vert\to 0$ as $k\to \infty$. Hence $w_{k}\to 0$ as $k\to \infty$.

To verify assertion (i), suppose that ${\bf(H_2)}$ holds while observing that the proof is similar under the fulfillment of ${\bf(H_3)}$. The R-continuity of $(\partial f)^{-1}$ at zero tells us that
$$
x_{k+1}\in (\partial f)^{-1}(w_{k+1})\subset (\partial f)^{-1}(0)+\rho(\Vert w_{k+1} \Vert)\ball
$$
which implies in turn that
$$
\mathbf{d}(x_{k+1},S)\le\rho(\Vert w_{k+1} \Vert)\to 0
$$
as $k\to \infty$, where $\rho$ stands for the modulus function of $(\partial f)^{-1}$ at zero.

To verify further assertion (ii), it follows from the boundedness of $(x_k)$ in finite dimensions that there exists a compact set $K$ containing $(x_k)$. By the imposed closed-graph property of $\partial f$, we deduce from  Theorem~\ref{compact} that $(\partial f)^{-1}$ is compactly R-continuous at zero. This gives us
$$
x_{k+1}\in (\partial f)^{-1}(w_{k+1})\cap K\subset (\partial f)^{-1}(0)+\rho(\Vert w_{k+1} \Vert)\ball
$$
and thus completes the proof of the theorem.
\end{proof}

\begin{remark}\label{clo-graph}{\rm$\,$ The following observations are in order:

{\bf(i)} Theorem~\ref{conver} allows us to find an approximate solution under conditions {\it not} acting on any stationary point $\bar{x}\in S$ as under the PLK conditions discussed above. In practice, $\bar{x}$ is unknown and is also calculated approximately by numerical algorithms. Therefore, it is important to find an approximate solution rather than an exact one, which is usually unavailable. This makes the results Theorem~\ref{conver} natural and reasonable for applications, especially when the convergence of iterative sequences $(x_k)$ is not available.

{\bf(ii)}  It is well known that the graph of the limiting subgradient mapping for {\it continuous} functions is always closed. For extended-real-valued l.s.c.\ functions $f$, the closed-graph property of \eqref{ls} has been established for broad classes of functions. It holds, in particular, for {\it prox-regular and subdifferential continuous} functions whose importance has been well recognized in variational analysis, optimization, and a variety of applications; see, e.g., \cite{Mordukhovich24,roc-wets,thibault}. Recall to this end that any {\it convex} extended-real-valued l.s.c.\ function is prox-regular and subdifferential continuous. The graph of the limiting subgradient mapping for the sum or difference of a continuously differentiable and a convex l.s.c. function is also closed.}
\end{remark}

There are many consequences of Theorem~\ref{conver} addressing particular algorithms of optimization. Let us formulate the one for the {\it Gradient Descent Method} (GDM), which satisfies the generic properties ${\bf(H_1)}$ and ${\bf(H_3)}$; see \cite{abs,Nesterov3}. As discussed above, BDCA is an example of algorithms in DC programming satisfying properties ${\bf(H_1)}$ and ${\bf(H_3)}$.

\begin{corollary}\label{boung}
Let  $f : \R^n \to\R$ be a differentiable function with $\inf_{x\in \R^n} f(x) >-\infty$, and let $(x_k)$ be the sequence of iterates generated by GDM. Suppose that the set of critical points $S:=(\nabla f)^{-1}(0)$ is 
nonempty. If $(x_k)$ is bounded, then $ \mathbf{d}(x_k,S)\to 0$ as $k\to \infty$.
\end{corollary}

Observe that Theorem~\ref{conver}(ii) assumes the boundedness of iterative sequences, which is not easy to verify in general. The next consequence of Theorem~\ref{conver} addresses PPA \eqref{opcopr} for which the required boundedness is guaranteed. 

\begin{corollary}
Given an l.s.c.\ convex  function $f : \R^n \to\overline{\R}$ with $\inf_{x\in \R^n} f(x) >-\infty$ and given the iterative sequence $(x_k)$ generated by PPA, assume that the solution set $S=(\partial f)^{-1}(0)$ is nonempty. Then we have $\mathbf{d}(x_k,S)\to 0$ as $k\to \infty$.
\end{corollary}
\begin{proof}
Let $\bar{x}\in S$. It is well known \cite{Rockafellar} that the sequence $(\Vert x_k-\bar{x}\Vert)$ is decreasing and convergent.  Thus $(x_k)$ is bounded and hence $ \mathbf{d}(x_k,S)\to 0$ as $k\to \infty$ by 
Theorem~\ref{conver}(ii). 
\end{proof}

\section{Algorithms of R-Class to Solve Inclusions}\label{sec:inc}

This section concerns the class of general inclusions \eqref{main} in finite-dimensional spaces, where the set-valued mapping $A: \R^n \rightrightarrows\R^n$ may not be of the subdifferential type. We show here that the (compact) 
R-continuity ideas proposed in the preceding section for optimization problems are useful to develop novel algorithms to solve general inclusion \eqref{main} and then specify them for the cases of maximally monotone and related operators ${\cal A}$.\vspace*{0.05in}

We say that an algorithm belongs to the {\it R-class} if the there exist a function $\xi: \N \to [0,\infty)$ with $\xi(k)\to 0\;{\rm as }\;k\to \infty$ and constants $\alpha, \beta >0$ such that for each $k=0,1,\ldots$ we have 
\begin{equation}\label{r-class}
\exists \;w_{k}\in  A (x_{k})\;\;{\rm with}\;\; \Vert w_{k} \Vert \le \alpha\xi^\beta(k).
\end{equation}

Roughly speaking, we want to designate a class of algorithms ensuring the existence of $w_{k}\in  A (x_{k})$ such that $w_{k}\to 0$. For the generic class of optimization algorithms considered above, we choose $\xi(k):=\Vert x_{k+1}-x_{k}\Vert$. Hence any such algorithm is in the R-class but the converse is not true even for optimization models. For example, consider the following algorithm to find stationary points of an l.s.c. function $f : \R^n \to\overline{\R}$ with $\inf_{x\in \R^n} f(x) >-\infty$ : given $x_k,\;k=0,1,\ldots$, select
\beq\label{opco}
x_{k+1}\in {\rm Argmin }_{x\in \R^n}\{f(x)+\gamma\Vert x-x_{k}\Vert^q\}
\eeq
for some $\gamma>0$ and $q>1$. It is not hard to check that algorithm \eqref{opco} belongs to the $R$-class with $A=\partial f$. Indeed, we readily get the inequality
$$
f(x_{k+1})+\gamma\Vert x_{k+1}-x_{k}\Vert^q\le f(x_k)\;\mbox{ for all }\;k=0,1,\ldots,
$$
which implies therefore that
$$
\sum_{k=0}^\infty\gamma\Vert x_{k+1}-x_{k}\Vert^q\le \sum_{k=0}^\infty f(x_k)-f(x_{k+1})\le f(x_0)-\inf_{x\in \R^n} f(x) <\infty
$$
and allows us to take $\xi(k)=\Vert x_{k+1}-x_{k}\Vert$. It follows from (\ref{opco}), by the subdifferential sum rule, that
$$
0\in \partial f(x_{k+1})+\gamma q \Vert x_{k+1}-x_{k}\Vert^{q-2}(x_{k+1}-x_{k}).
$$
Letting $w_{k+1}:=-\gamma q \Vert x_{k+1}-x_{k}\Vert^{q-2}(x_{k+1}-x_{k})$ tells us that
$$w_{k+1}\in  A (x_{k+1})\;\mbox{ and }\;
\Vert w_{k+1} \Vert \le\gamma q \Vert x_{k+1}-x_{k}\Vert^{q-1}
$$
from which we deduce \eqref{r-class} and thus verify that \eqref{opco} is an algorithm of R-class.

Note that if $q=1$ and $\gamma$ is small in \eqref{opco}, then $\Vert w_{k+1}\Vert$ is small as well and the iterative sequence $(x_k)$ is convergent due to $\sum_{k=0}^\infty\Vert x_{k+1}-x_{k}\Vert<\infty$. This explains why the algorithm 
\begin{equation*}
x_{k+1}\in {\rm Argmin }_{x\in \R^n}\{f(x)+\gamma\Vert x-x_{k}\Vert\}
\end{equation*}
with $\gamma>0$ sufficiently small can be also used to find an approximate solution of \eqref{main}. 

Observe that if we expect the convergence of the iterative sequence $(x_k)$ to a solution of \eqref{main},  then the convergence of of the sequence ($\Vert x_{k+1}-x_{k}\Vert$) to zero is unavoidable. On the other hand, the R-class can contain  algorithms for finding approximate solutions {\it without convergence} of iterates as, e.g., in Nesterov's Accelerated Gradient Method \cite{Nesterov3}. If in addition the R-continuity of the solution mapping at zero is provided, then obtaining an approximate solution  of \eqref{main} is guaranteed. Thus we first check if an algorithm is of R-class and then verify the mapping R-continuity.\vspace*{0.05in}

The next theorem ensures the global convergence of generic algorithms of R-class under the R-continuity and 
compact R-continuity assumptions.

\begin{theorem}\label{converg}
Given a set-valued mapping $\mathcal{A}: \R^n \rightrightarrows \R^n$ in \eqref{main}, consider the sequence of iterates $(x_k)$ generated by an algorithm of R-class and suppose that the solution set $S:=\mathcal{A}^{-1}(0)$ is nonempty. The following assertions hold:

{\bf(i)} If $\mathcal{A}^{-1}$ is R-continuous at zero, then $ \mathbf{d}(x_k,S)\to 0$ as $k\to \infty$.

{\bf(ii)}  If the sequence $(x_k)$ is bounded and the mapping $\mathcal{A}^{-1}$ is compactly R-continuous at zero, then  we have $\mathbf{d}(x_k,S)\to 0$ as $k\to\infty$.
\end{theorem}
\begin{proof}
To verify assertion (i), deduce from the R-continuity of $\mathcal{A}^{-1}$ at zero that
$$
x_{k}\in \mathcal{A}^{-1}(w_{k})\subset A^{-1}(0)+\rho(\Vert w_{k} \Vert)\ball.
$$
This readily implies the claimed convergence
$$
\mathbf{d}(x_{k},S)\le\rho(\Vert w_{k} \Vert)\to 0
$$
as $k\to \infty$, where $\rho$ denotes the modulus function $\mathcal{A}^{-1}$ at zero.

To check the fulfillment of (ii), we find by the boundedness of $(x_k)$  a compact set $K$ such that $(x_k)$ is contained in $K$. Since $\mathcal{A}^{-1}$ is compactly R-continuous at zero, this gives us  
$$
x_{k}\in \mathcal{A}^{-1}(w_{k})\cap K\subset \mathcal{A}^{-1}(0)+\rho(\Vert w_{k} \Vert)\ball,
$$
which therefore completes the proof of the theorem.
\end{proof}

Here is a  consequence of Theorem~\ref{converg}(i) ensuring the fulfillment of the R-continuity assumption.

\begin{corollary}
Let $\mathcal{A}: \R^n \rightrightarrows \R^n$ be such that $\mathcal{A}$ has closed graph and $\mathcal{A}^{-1}$ is { locally bounded}  around zero, and  let $(x_k)$ be the sequence generated by an R-class algorithm. Supposing that the solution set $S=\mathcal{A}^{-1}(0)$ is nonempty, we have $\mathbf{d}(x_k,S)\to 0$ as $k\to \infty$.
\end{corollary}
\begin{proof}
The graph of $\mathcal{A}^{-1}$ is closed due the closed-graph assumption on ${\cal A}$. The additional assumption that the mapping $\mathcal{A}^{-1}$ is locally bounded around zero yields the R-continuity of $\mathcal{A} ^{-1}$ at this point; see \cite{L1} for more details. Thus the conclusion follows from Theorem~\ref{converg}.
\end{proof}

Now we consider inclusion \eqref{main} generated by a {\it maximally  monotone} operator and establish the global convergence of the Proximal Point Algorithm (PPA) to solve such inclusions by employing  the compactly 
R-continuity result  of Theorem~\ref{converg}(ii). Note that the PPA in the theorem below is an inclusion version of PPA \eqref{opcopr} considered in Section~\ref{s4} for convex optimization problems. As a consequence, we can have an alternative approach to obtain the convergence of the (PPA). This confirms the consistency of our theoretical results  developed  above.

\begin{theorem}\label{maxi}
Let $\mathcal{A}: \R^n \rightrightarrows \R^n$ be a  maximally  monotone operator, and let $(x_k)$ be the iterative sequence generated by the PPA:
\begin{equation}\label{ppa}
x_0\in \R^n, \;x_{k+1}=J_{\gamma \mathcal{A}}(x_k)\;\mbox{ for }\;k=0,1,\ldots,
\end{equation}
where $J_{\gamma \mathcal{A}}:=(\gamma \mathcal{A}+I)^{-1}$ denotes the resolvent of $\gamma \mathcal{A}$, and where  $I$ stands for the identity operator. If the solution set $S:=\mathcal{A}^{-1}(0)$ is nonempty, then
 $\mathbf{d}(x_k,S)\to 0$ as $k\to \infty$.  Moreover, the iterative sequence $(x_k)$ converges to some $\bar x\in S$.
\end{theorem}
\begin{proof}
First we check that $(x_k)$ is bounded and that $(\Vert x_{k+1}-x_{k}\Vert)$ converges to zero; cf.\ \cite{Rockafellar}. Indeed, picking $\bar x\in S$ gives us the inclusions
$$
\frac{x_{k+1}-x_k}{\gamma} \in - \mathcal{A}(x_{k+1})\;\mbox{ and }\;0\in \mathcal{A}(\bar x).
$$
If follows from the monotonicity of $\mathcal{A}$ that
$$
\langle x_{k+1}-\bar x, x_{k+1}-x_k\rangle\le 0, 
$$
which yields in turn the relationships
$$
\Vert x_{k+1}-\bar x\Vert^2\le \langle x_{k+1}-\bar x,  x_{k}-\bar x\rangle=\frac{1}{2}(\Vert x_{k+1}-\bar x\Vert^2+\Vert x_{k}-\bar x\Vert^2-\Vert x_{k+1}-x_{k}\Vert^2).
$$
Therefore, we arrive at the estimate
$$
\Vert x_{k+1}-\bar x\Vert^2\le \Vert x_{k}-\bar x\Vert^2-\Vert x_{k+1}-x_{k}\Vert^2,
$$
which ensures the decreasing property and convergence of $(\Vert x_{k+1}-\bar x\Vert)$. Hence the sequence $(\Vert x_{k+1}-x_k\vert)$ converges to zero and $(x_k)$ is bounded. Since the graph of $\mathcal{A}^{-1}$ is closed, the mapping $\mathcal{A}^{-1}$ is compactly R-continuous at zero, and the conclusion follows from Theorem~\ref{converg}(ii). 

To verify the second claim, deduce from the boundedness of $(x_k)$ that there exists a subsequence $(x_{m_k})$ that converges to some $\bar x$. Since $\mathcal{A}^{-1}$ is a maximally monotone operator, the set $S=\mathcal{A}^{-1}(0)$ is closed and convex; see, e.g., \cite{br}. Noting that the distance function $\mathbf{d}(\cdot,S)$ is continuous and $\mathbf{d}(x_{m_k},S)\to 0$, we get that $\mathbf{d}(\bar x,S)=0$, or equivalently $\bar x\in S$. Since the sequence $(\Vert x_{k}-\bar x\Vert)$ is decreasing and convergent, we justify the convergence of $(x_k)$ converges to  $\bar x$.
\end{proof}

The next theorem does not require the maximal monotonicity of ${\cal A}$, while replacing this by the maximal monotonicity of the shifted operator ${\cal A}^{-1}+\kappa I$.

\begin{theorem}\label{shifted}
Let $\mathcal{A}: \R^n \rightrightarrows \R^n$ be such that the shifted operator $\mathcal{A}^{-1}+\kappa I$ is a  maximally monotone for some  $\kappa>0$, and let $(x_k)$ be the sequence generated by PPA \eqref{ppa} with $\gamma<(2\kappa)^{-1}$. Then we have $\mathbf{d}(x_k,S)\to 0$ as $k\to \infty$ for $S={\cal A}^{-1}(0)$.
\end{theorem}
\begin{proof}
Taking $\bar x\in S$ gives us the inclusions
$$
\frac{x_{k+1}-x_k}{\gamma} \in - \mathcal{A}(x_{k+1})\;\mbox{ and }\;0\in \mathcal{A}(\bar x).
$$
We deduce from the monotonicity of $\mathcal{A}^{-1}+\kappa I$ the equivalence
\baqn
&&\Big\langle \frac{x_{k+1}-x_k}{\gamma}, x_{k+1}-\bar x\Big\rangle\le \kappa\Big\Vert\frac{x_{k+1}-x_k}{\gamma}\Big\Vert^2\\
&\iff& \Vert x_{k+1}-\bar x\Vert^2\le \Vert x_{k}-\bar x\Vert^2-(1-2 \gamma \kappa)\Vert x_{k+1}-x_{k}\Vert^2.
\eaqn
The assumption of $\gamma<(2\kappa)^{-1}$ means that $2 \gamma \kappa<1$, which yields the decreasing property and convergence  of the sequence $(\Vert x_{k+1}-\bar x\Vert)$. Consequently, the sequence $(\Vert x_{k+1}-x_{k}\Vert)$ converges to zero and $(x_k)$ is bounded. Since the graph of $\mathcal{A}^{-1}+\kappa I$ is closed by the maximal monotonicity, we have this also for the the graph of $\mathcal{A}^{-1}$. Thus the operator $A^{-1}$ is compactly R-continuous at zero, and the conclusion of the theorem follows from Theorem~\ref{converg}(ii). 
\end{proof}

The last theorem of this section addresses problems of { DC programming} and presents a version of {\it DCA} for which the compact R-continuity allows us to justify the global convergence of approximate solutions. The reader can consult with \cite{Leth} and the references therein for more information on DCA and related algorithms to solve DC programs.

\begin{theorem}\label{dca}
Suppose that $f=g-h$ is bounded below, where $g\colon\R^n\to\overline{\R}$ are proper l.s.c. convex function, and where $h\colon\R^n\to\R$ is is convex and continuously differentiable.   Let $(x_k)$ be the sequence of iterates generated by the following version of DCA:
\beq\label{xk}
x_0\in H,\;x_{k+1}= J_{\gamma\partial g}L_{-\gamma\nabla h}(x_k),
\eeq
where $J_{\gamma\partial g}$ is the resolvent of $\gamma\partial g$, and where $L_{-\gamma\nabla h}:=I+\gamma \nabla h.$  Assume that the solution set $S=(\partial f)^{-1}(0)$ is nonempty and that $(x_k)$ is bounded. Then we have 
$\mathbf{d}(x_k,S)\to 0$ as $k\to \infty$.
\end{theorem}
\begin{proof}
Observe first that the graph of $\partial f$ is closed in this setting; see Remark~\ref{clo-graph}(ii). Thus $(\partial f)^{-1}$ is compactly R-continuous at zero by Theorem~\ref{compact}. Following the proof of \cite[Theorem~4.5]{L1}, we get
\begin{equation*}
\lim_{k\to+\infty} \Vert x_{k+1}-x_k\Vert=0,
\end{equation*}
$$
\nabla h(x_k)-\nabla h(x_{k+1})-\frac{x_{k+1}-x_k}{\gamma}\in \partial g(x_{k+1})-\nabla h(x_{k+1}).
$$
Consider further the sequence of vectors
$$
r_k:=\nabla h(x_k)-\nabla h(x_{k+1})-\frac{x_{k+1}-x_k}{\gamma},\quad k=0,1, 2, \ldots.
$$
Since $(x_k)$ is bounded, there exists a compact set $K$ containing $(x_k)$. The gradient mapping  $\nabla h$ is uniformly continuous on $K$, which implies that $r_k\to 0$ as $k\to\infty$ and that
$$
x_{k+1}\in  (\partial g - \nabla h)^{-1}(r_k)\cap K \subset (\partial g - \nabla h)^{-1}(0) +\rho(\Vert r_k \Vert),
$$
where $\rho$ is the modulus function of $(\partial f)^{-1}$ at zero. Therefore, 
$$
\lim_{k\to \infty}{ \mathbf{d}}(x_{k+1},S)=0,
$$
which completes the proof of the theorem.
\end{proof}

\begin{remark} $\;$ {\rm  Our final remarks in this section are as follows:

{\bf(i)} Theorem \ref{dca}, based on the compact R-continuity,  is a new result in the study of DCA, which explains {\it why DCA is efficient} for a large class of operators in DC programming. Usually the boundedness of $(x_k)$ only allows us to ensure that any limiting point of $(x_k)$ is  a critical point of $f$; see, e.g., \cite{Leth}. Note that if the solution set is locally bounded and $\nabla h$ is uniformly continuous (e.g., when $ h$ is a quadratic function), then the boundedness assumption of  $(x_k)$ can be omitted.

{\bf(ii)} If $\nabla h$ is uniformly continuous, then the version of  DCA in Theorem~\ref{dca} is an {\it R-class algorithm} with the available compact R-continuity property.}
\end{remark}

\section{Conclusions and Future Research}\label{s5}

In this paper, we introduce and study the notion of compact R-continuity for set-valued mappings, which is a relaxed version of R-continuity, in both finite-dimensional and infinite-dimensional spaces. Characterizations and other important properties of compactly R-continuous operators are revealed with establishing relationships between this notion and the classical \L ojasiewicz's inequality for analytic functions. Our special attention is paid to  applications of R-continuity and its compact version to the design and justification of numerical algorithms in constrained optimization and solving inclusion problems in finite-dimensional spaces. We show that the (compact) R-continuity offers an alternative approach to the global convergence of generic algorithms in nonsmooth, which have some advantages over the conventional approach based on the Polyak-\L ojasiewicz-Kurdyka conditions. These ideas are extended to solving inclusion problems via algorithms of R-class, which may not be related to finding stationary points in optimization.

There are many topics of future research in the directions of this paper. First of all, we plan to further investigate numerical aspects of R-continuity and compact R-continuity in optimization-related areas with deriving convergence rates of the corresponding algorithms. Our goals also include establishing more detailed relationships between the (compact) R-continuity of subgradient mappings and the aforementioned PLK conditions in optimization. In addition, we intend to apply the obtained numerical results to special classes of algorithms particularly important in  constrained optimization and to extend the numerical results to problems in infinite-dimensional spaces.

\end{document}